\documentclass[final,3p,times]{elsarticle}

\usepackage{amssymb, hyperref, amsmath, algorithm, algpseudocode,amsthm}
\usepackage[T1]{fontenc}
\biboptions{sort&compress}
\hypersetup{hidelinks}
\hypersetup{
colorlinks=true,
linkcolor=black
}

\newcommand{\be}{\begin{equation}}
\newcommand{\ee}{\end{equation}}
\newcommand{\bex}{\begin{equation*}}
\newcommand{\eex}{\end{equation*}}
\newcommand{\bea}{\begin{eqnarray}}
\newcommand{\eea}{\end{eqnarray}}
\newcommand{\beas}{\begin{eqnarray*}}
\newcommand{\eeas}{\end{eqnarray*}}

\newtheorem{theorem}{Theorem}[section]
\newtheorem{definition}[theorem]{Definition}
\newtheorem{proposition}[theorem]{Proposition}
\newtheorem{lemma}[theorem]{Lemma}

\newtheorem{remark}[theorem]{Remark}
\newtheorem{assumption}{Assumption}[section]
\newproof{pf}{Proof}

\makeatletter
\newenvironment{breakablealgorithm}
  {
   \begin{center}
     \refstepcounter{algorithm}
     \hrule height.8pt depth0pt \kern2pt
     \renewcommand{\caption}[2][\relax]{
       {\raggedright\textbf{\ALG@name~\thealgorithm} ##2\par}%
       \ifx\relax##1\relax 
         \addcontentsline{loa}{algorithm}{\protect\numberline{\thealgorithm}##2}%
       \else 
         \addcontentsline{loa}{algorithm}{\protect\numberline{\thealgorithm}##1}%
       \fi
       \kern2pt\hrule\kern2pt
     }
  }{
     \kern2pt\hrule\relax
   \end{center}
  }
\makeatother

\journal{journal}

\begin{document}

\begin{frontmatter}



\title{Inverse reinforcement learning by expert imitation for the stochastic linear-quadratic optimal control problem}


\author[1]{Zhongshi Sun}
\ead{stone.sun@mail.sdu.edu.cn}

\author[1]{Guangyan Jia\corref{cor1}}
\ead{jiagy@sdu.edu.cn}

\cortext[cor1]{Corresponding author}

\affiliation[1]{organization={Zhongtai Securities Institute for Financial Studies, Shandong University},
            city={Jinan},
            postcode={250100}, 
            state={Shandong},
            country={China}}
\begin{abstract}
This article studies inverse reinforcement learning (IRL) for the stochastic linear-quadratic optimal control problem, where two agents are considered. A learner agent does not know the expert agent's performance cost function, but it imitates the behavior of the expert agent by constructing an underlying cost function that obtains the same optimal feedback control as the expert's. We first develop a model-based IRL algorithm, which consists of a policy correction and a policy update from the policy iteration in reinforcement learning, as well as a cost function weight reconstruction based on the inverse optimal control. Then, under this scheme, we propose a model-free off-policy IRL algorithm, which does not need to know or identify the system and only needs to collect the behavior data of the expert agent and learner agent once during the iteration process. Moreover, the proofs of the algorithm's convergence, stability, and non-unique solutions are given. Finally, a simulation example is provided to verify the effectiveness of the proposed algorithm.
\end{abstract}




\end{frontmatter}


\section{Introduction}
\label{sec:intro}
In recent years, significant progress has been made in solving optimal control problems with unknown models using reinforcement learning (RL) \citep{bertsekas1996neuro,sutton2018reinforcement} and approximate/adaptive dynamic programming (ADP) \citep{lewis2013reinforcement,jiang2017robust}. These methods learn optimal policies directly through interaction with the environment, without the need for model parameter information, opening up new frontiers in control theory. For deterministic systems, an ADP algorithm was proposed by \citet{vrabie2009adaptive} for solving a continuous-time linear quadratic optimal control problem (LQ problem) without knowing information on internal system dynamics. \citet{jiang2012computational} presented a novel policy iteration (PI) method for the continuous-time LQ problem with completely unknown system dynamics. Afterward, the robustness property of PI in RL was discussed in \cite{pang2021robust,cui2023lyapunov}. Moreover, more RL algorithms for optimal control problems for linear/nonlinear/continuous-time/discrete-time systems can be found in \citep{liu2020adaptive,jiang2020learning,wang2023recent}. On the other hand, for stochastic systems, \citet{bian2016adaptive} presented ADP and robust ADP algorithms for a class of continuous-time stochastic systems subject to multiplicative noise. \citet{pang2022reinforcement} developed a PI-based off-policy RL algorithm for the adaptive optimal stationary control of continuous-time linear stochastic systems with both additive and multiplicative noises. \citet{wang2020reinforcement} proposed an entropy-regularized, exploratory control framework to investigate the exploration-exploitation tradeoff for RL in the stochastic LQ case. \citet{li2022stochastic} presented an online RL algorithm to solve the stochastic LQ problem with partial system information, and \citet{zhang2023adaptive} extended it to an ADP-based completely model-free PI algorithm. 

The advancement of RL/ADP has subsequently given rise to the exploration of inverse problems, specifically inverse reinforcement learning (IRL) \citep{ng2000algorithms,abbeel2004apprenticeship} and inverse optimal control (IOC) \citep{bellman1963inverse,ab2020inverse}. These fields address the challenge of uncovering the underlying objective functions or costs that govern observed optimal behaviors. While the traditional RL method focuses on deriving an optimal policy with a predefined reward function, IRL and IOC study the inverse problem, determining what the reward or cost function might be based on observations of the expert agent's performance and assuming that such behavior is optimal. Numerous applications of IRL theory have been found in many areas, such as autonomous driving \citep{you2019advanced}, anomaly detection \citep{oh2019sequential}, energy management \citep{tang2022multi}, and so on.

Exploration into IRL and IOC has been progressively growing, with a variety of methods being proposed to tackle these inverse problems \citep{lian2024integral}. A novel IRL algorithm was presented by \citet{xue2021inverse} for learning an unknown performance objective function for tracking control. \citet{lian2022data} developed a data-driven IRL control algorithm for nonzero-sum multiplayer games in linear continuous-time differential dynamical systems. \citet{self2022model} proposed an online data-driven model-based IRL technique for linear and nonlinear deterministic systems. There are also many IRL researches in \citep{lian2022inverse,donge2022multi,lian2023off,martirosyan2023inverse,wu2023human,lian2021online,xue2023data,zhang2024inverse}. However, the methods mentioned above are designed for deterministic systems, and limited work has been done for stochastic systems. \citet{karg2022inverse} defined and solved the stochastic IOC problem of the discrete-time LQ Gaussian and the LQ sensorimotor control model. \citet{li2023inverse} studied the stochastic IOC problem for discrete-time LQ Gaussian
tracking control. \citet{nakano2023inverse} discussed an inverse problem of the stochastic optimal control of general diffusions with the performance index having the quadratic penalty term of the control process.

In this article, we propose a novel IRL approach for the continuous-time stochastic LQ optimal control problem that learns both underlying cost weight and the corresponding optimal control policy. Compared with the IRL algorithm in \cite{xue2021inverse} for linear deterministic systems, this article considers a class of linear It\^{o} stochastic systems with state and control-dependent diffusion terms. Firstly, we consider a target expert agent and a learner agent and present a model-based IRL method for the learner to imitate the target expert agent's behaviors by constructing an unknowing cost function that obtains the expert's feedback control. This algorithm consists of a policy correction, a policy update using the standard RL PI, and a cost function weight construction using the standard IOC. Secondly, rigorous convergence proof and analysis of stability of this model-based IRL algorithm are provided. Moreover, since the cost weights that can generate the same control policy in IRL may not be unique, we characterize the set of all possible solutions. Then, we develop a model-free off-policy IRL algorithm based on the model-based algorithm. This IRL algorithm does not need to know or identify the system and only needs to collect the behavior data of the target expert agent and learner agent once during the iteration process. Finally, we conduct a numerical experiment on a 2-dimensional linear system based on the proposed model-free algorithm. The simulation results show that the algorithm can converge to the optimal solution well.

The rest of the paper is organized as follows. In Section 2, we briefly describe target optimal control, learner dynamics, and IRL control problems. Section 3 presents a model-based IRL algorithm for the stochastic LQ problem and analyzes the convergence, stability, and nonuniqueness of the algorithm. Then, we develop a model-free off-policy IRL algorithm for the stochastic LQ problem in Section 4. In Section 5, we implement the model-free algorithm and present a simulation result.

Notations: Let $\mathbb{R}^{n\times m}$ be the set of all $n \times m$ real matrices. Let $\mathbb{R}^{n}$ be the $n$-dimensional Euclidean space. For a matrix $X \in \mathbb{R}^{n \times m}$, $vec(X)$ = $\begin{aligned}\begin{bmatrix} x_{1}^{\top}, x_{2}^{\top}, \cdots x_{m}^{\top}\end{bmatrix}^{\top}\end{aligned}$, where $x_{i}^{\top} \in \mathbb{R}^{n}$ is the $i$-th column of $X$. $\otimes$ indicates the Kronecker product. $\mathbb{S}^{n}$ is the set of all symmetric matrices in $\mathbb{R}^{n \times n}$. $\Vert \cdot \Vert$ is the matrix norm in $\mathbb{R}^{n\times m}$. $I_{n}$ denotes the $n$-dimensional identity matrix.

\section{Problem formulation and some preliminaries}

In this section, we consider two dynamic agents: the target expert agent and the learner agent. The former shows optimal examples related to expert performance cost functions. The latter attempts to learn the underlying cost function from expert examples and imitate its behavior. The learner agent can only observe the control actions and state behaviors of the target agent but does not know the information about its system dynamics and performance cost function.

\subsection{Target optimal control}
Let $(\Omega, \mathcal{F}, \mathbb{F}, \mathbb{P})$ be a complete filtered probability space on which a one-dimensional standard Brownian motion $W(\cdot)$ is defined, with $\mathbb{F}=\left\{\mathcal{F}_t\right\}_{t \geqslant 0}$ being its natural filtration augmented by all the $\mathbb{P}-$null sets in $\mathcal{F}$. Consider a target expert agent whose trajectory is generated by the following time-invariant stochastic linear dynamical control system:
\be
\label{sde}
\left\{\begin{array}{ll}
dX_{T}(s) = \left[AX_{T}(s)+Bu_{T}(s)\right]d s + \left[CX_{T}(s)+Du_{T}(s)\right]d W(s), \\
X_{T}(t)=x,
\end{array}\right.
\ee
where $X_{T}\left(\cdot\right) \in \mathbb{R}^{n}$ is the target state, and $u_{T}(\cdot) \in \mathbb{R}^{m}$ is the target control. The coefficients $A$, $C \in \mathbb{R}^{n \times n}$, and $B$, $D \in \mathbb{R}^{n \times m}$ are constant matrix. Let $\mathbb{H}$ be a Euclidean space, and we define the following space:
$L_{\mathbb{F}}^{2}(\mathbb{H})=\left\{\varphi:[0, \infty) \times \Omega \rightarrow \mathbb{H} \mid \varphi(\cdot)\right.$ is an $\mathbb{F}$-progressively measurable process, and $\mathbb{E} \int_{0}^{t}|\varphi(s)|^{2} d s<\infty$ for each $t \ge 0\left.\right\}$.

Clearly, the state equation \eqref{sde} admits a unique solution $X_{T}$ for any control $u_{T} \in L_{\mathbb{F}}^{2}(\mathbb{R}^{m})$. 

In addition, we consider a quadratic target performance cost functional given by
\be
\label{per_func}
J_{T}\left(t, x; u(\cdot)\right)=\mathbb{E}^{\mathcal{F}_{t}}\int_t^{\infty}\left(X_{T}(s)^{\top} Q_{T} X_{T}(s)+u_{T}(s)^{\top} R_{T} u_{T}(s)\right) \mathrm{d} s ,
\ee
where $Q_{T} =Q_{T}^{\top} > 0 $ and $R_{T} =R_{T}^{\top} > 0$ are given constant matrices of proper sizes.

In order to avoid the cost functional \eqref{per_func} tending to positive infinity, we give the following definition of mean-square stabilizability and assumption.

\begin{definition}
The system \eqref{sde} is called mean-square stabilizable if there exists a constant matrix $K \in \mathbb{R}^{m \times n}$ such that the solution of
\be
\left\{\begin{array}{ll}
dX_{T}(s) = (A+BK)X_{T}(s)d s + (C+DK)X_{T}(s)d W(s), \\
X_{T}(t)=x,
\end{array}\right.
\ee
satisfying $\lim_{s \rightarrow \infty} \mathbb{E}[X_{T}(s)^{\top}X_{T}(s)] =0$. In this case, $K$ is called a stabilizer of the system \eqref{sde}, and the feedback control $u_{T}(\cdot) = KX_{T}(\cdot)$ is called stabilizing.
\end{definition}

\begin{assumption}
\label{ass_1}
System \eqref{sde} is mean-square stabilizable.
\end{assumption}

Under Assumption \eqref{ass_1}, we define the sets of admissible controls as
\bex
\mathcal{U}_{a d}=\left\{u(\cdot) \in L_{\mathbb{F}}^2\left(\mathbb{R}^{m}\right) \mid u(\cdot) \text { is stabilizing }\right\} \text {. }
\eex

In the stochastic LQ problem, the target control $u_{T}$ seeks to minimize \eqref{per_func}. The goal is to find an optimal target control $u_{T}^{\ast}(\cdot) \in \mathcal{U}_{a d}$ such that:
\be
\label{value}
J_{T}(t, x ; u_{T}^{\ast}(\cdot))=\inf _{u_{T}(\cdot) \in L_{\mathbb{F}}^{2}(\mathbb{R}^{m})} J_{T}(t, x ; u_{T}(\cdot)) \triangleq V_{T}(t, x).
\ee

We call $V_{T}(t,x)$ the target value function of the stochastic LQ problem. The optimal target control $u_{T}^{\ast}$ is a feedback control scheme of the following form:
\be
\label{opti_con}
u_{T}^{\ast}(s) = K_{T}X(s) = -\left(R_{T}+D^{\top}P_{T}D\right)^{-1}\left(B^{\top}P_{T} + D^{\top}P_{T}C\right)X(s),
\ee
and target value function $V_{T}(t,x) = x^{\top}P_{T}x$, where $P_{T} =P_{T}^{\top} > 0 $ is the unique solution to the following target stochastic algebraic Riccati equation (SARE) \citep{sun2020stochastic}:
\be
\label{ricc}
P_{T} A+A^{\top} P_{T}+C^{\top} P_{T} C+Q_{T}-\left[P_{T} B+C^{\top} P_{T} D\right]\left[R_{T}+D^{\top} P_{T} D\right]^{-1}\left[B^{\top} P_{T}+D^{\top} P_{T} C\right]=0.
\ee

\subsection{Learner dynamics and IRL control problem}
 The learner trajectory is generated by the following time-invariant stochastic linear dynamical control system:
\be
\label{sdel}
\left\{\begin{array}{ll}
dX(s) = \left[AX(s)+Bu(s)\right]d s + \left[CX(s)+Du(s)\right]d W(s), \\
X(t)=x,
\end{array}\right.
\ee
where $X\left(\cdot\right) \in \mathbb{R}^{n}$ is the learner's state, and $u(\cdot) \in \mathbb{R}^{m}$ is the learner's control.

Let us introduce the stochastic LQ problem of the learner \eqref{sdel} with an arbitrary given cost functional
\be
\label{per_funcl}
J\left(t, x; u(\cdot)\right)=\mathbb{E}^{\mathcal{F}_{t}}\int_t^{\infty}\left(X(s)^{\top} Q X(s)+u(s)^{\top} R u(s)\right) \mathrm{d} s ,
\ee
where $Q =Q^{\top} > 0 $ and $R =R^{\top} > 0$ are given constant matrices of proper sizes. The optimal input $u$ is given by
\be
\label{opti_conl}
u(s) = KX(s) = -\left(R+D^{\top}PD\right)^{-1}\left(B^{\top}P + D^{\top}PC\right)X(s)
\ee
where $P =P^{\top} > 0 $ is the solution of the learner SARE :

\be
\label{riccl}
P A+A^{\top} P+C^{\top} P C+Q-\left[P B+C^{\top} P D\right]\left[R+D^{\top} P D\right]^{-1}\left[B^{\top} P+D^{\top} P C\right]=0.
\ee

We now give two assumptions and a definition for discussing the IRL problem.
\begin{assumption}
The weights $Q_{T}$ and $R_{T}$ in the target performance cost functional $J_{T}$ \eqref{per_func}, the optimal gain matrix $K_{T}$ in \eqref{opti_con}, and SARE solution $P_{T}$ are unknown to learner \eqref{sdel}.
\end{assumption}

\begin{assumption}
The learner knows the target state data $X_{T}$ and control data $u_{T}$.
\end{assumption}

\begin{definition}[\citep{lian2024integral}]
Given an $R \in \mathbb{R}^{m \times m}>0$, if there exists a $Q\in \mathbb{R}^{n \times n} \geq 0$ in \eqref{riccl}, such that $K = K_{T}$, then $Q$ is called an equivalent weight to $Q_{T}$.
\end{definition}

Then, based on the IRL work \citep{xue2021inversediscrete, xue2021inverse}, we give the following definition of the IRL for the stochastic LQ problem.
\begin{definition}[IRL for stochastic LQ problem]
\label{df_IRL}
By using the collected data of the expert \eqref{sde} and the learner \eqref{sdel}, the learner desires to design an IRL problem that holds the following:
\begin{enumerate}[(i)]
\item Reconstruct the unknown performance cost functional \eqref{per_func} to exhibit the same control actions $u_{T}^{\ast}$ in \eqref{opti_con} and states $x_{T}$ as the expert \eqref{sde}, i.e., find an equivalent weight $Q$ to $Q_{T}$ for any $R \in \mathbb{R}^{m \times m}>0$.
\item Make sure that each computed cost function in the whole iteration process stabilizes the learner agent \eqref{sdel}.
\end{enumerate}
\end{definition}

\section{Model-based IRL framework for the stochastic LQ problem}
In this section, we propose a novel model-based IRL approach for the learner agent \eqref{sdel} to determine an equivalent weight to $Q_{T}$ by using optimal control and IOC methods. This approach requires knowledge of the system dynamics, and in Section 4, we will propose a model-free algorithm to remove this requirement.

\subsection{IRL policy iteration}

First, we present a theorem to show the conditions that the solutions to the IRL problem in Definition \ref{df_IRL} must satisfy.

\begin{theorem}
\label{th1}
If the cost weight $Q,R$ and the solution $P$ satisfy both the learner SARE \eqref{riccl} and the following equation
\be 
\label{lya_KT}
 P\left(A+BK_{T}\right) +\left(A+BK_{T}\right)^{\top}P+ \left(C + DK_{T}\right)^{\top}P\left(C + DK_{T}\right) + Q + K_{T}^{\top}RK_{T} = 0,
\ee
then, the corresponding gain matrix $K$ in \eqref{opti_conl} equals the target gain matrix $K_{T}$ in \eqref{opti_con}.
\end{theorem}
\begin{proof}
Rewrite \eqref{riccl} with $K_{T}$ in \eqref{opti_con} and $K$ in \eqref{opti_conl} as 
\bex
\begin{aligned}
&P\left(A + BK_{T}\right) + \left(A + BK_{T}\right)^{\top}P+  \left(C + DK_{T}\right)^{\top}P\left(C + DK_{T}\right) + K_{T}^{\top}\left(R+D^{\top}PD\right)K \\ &+ K^{\top}\left(R+D^{\top}PD\right)K_{T} - K_{T}^{\top}D^{\top}PDK_{T} - K^{\top}\left(R+D^{\top}PD\right)K + Q = 0.
\end{aligned}
\eex
and subtract it from \eqref{lya_KT} to yield
\be 
\label{KT-K}
\left(K_{T} - K\right)^{\top}\left(R+D^{\top}PD\right)\left(K_{T} - K\right) = 0.
\ee

Since $R>0$, \eqref{KT-K} concludes that $K = K_{T}$. This completes the proof.
\end{proof}

Then, since the data of target trajectory $X_{T}$ and $u_{T}$ can be measured, we use collected data to obtain the expert's gain matrix $K_{T}$ in \eqref{opti_con}. Write \eqref{opti_con} expanded as 
\be 
\label{ls_KT}
\begin{aligned}
\tilde{u}_{T}:=&\left[u_{T}(t-(k-1)\Delta t),~\dots,~u_{T}(t-\Delta t),~u_{T}(t)\right] \\=& K_{T}\left[X_{T}(t-(k-1)\Delta t),~\dots,~X_{T}(t-\Delta t),~X_{T}(t)\right] =: K_{T}\tilde{X}_{T},
\end{aligned}
\ee
where the integer $k \geq mn$ is the number of data groups of $X_{d}$ and $u_{d}$. By using the batch least-square method to solve  \eqref{ls_KT}, the estimate of $K_{T}$ is given by
\be 
\label{ls_KT2}
\hat{K}_{T} = \tilde{u}_{T}\tilde{X}_{T}^{\top}\left(\tilde{X}_{T}\tilde{X}_{T}^{\top}\right)^{-1}.
\ee 

Next, assuming the effectiveness of least-square methods has $\hat{K}_{T} = K_T$, an IRL policy iteration procedure to find the $Q$, $R$, and $P$ satisfying Theorem \ref{th1} is summarized in the following algorithm.

\begin{breakablealgorithm}
  \label{algo1}
  \caption{Model-based IRL algorithm for the stochastic LQ problem.} 
  \renewcommand{\algorithmicrequire}{\textbf{Initialization:}}
\begin{algorithmic}[1]
  \Require Choose an arbitrary matrix $R >0$, an initial reward weight matrix $Q^{0}>0$, and compute $\hat{K}_{T}$ by \eqref{ls_KT2}.
  \State Set $i=0$.
  \Repeat
  \State (Policy correction) Update $P^{i}$ by
  \be 
  \label{lya_algo1}
  P^{i}\left(A+BK_{T}\right) +\left(A+BK_{T}\right)^{\top}P^{i}+ \left(C+DK_{T}\right)^{\top}P^{i}\left(C+DK_{T}\right) = -Q^{i} - K_{T}^{\top}RK_{T}.
  \ee
  \State (Policy update) Update control policy based on 
  \be
  \label{con_algo1}
  u^{i+1} = K^{i+1}x = -\left(R+D^{\top}P^{i}D\right)^{-1}\left(B^{\top}P^{i} + D^{\top}P^{i}C\right)x.
  \ee
  \State (Cost function weight construction) Update $Q^{i+1}$ by
  \be 
  \label{q_algo1}
  Q^{i+1} = -A^{\top}P^{i}- P^{i}A - C^{\top}P^{i}C + (K^{i+1})^{\top}\left(R+D^{\top}P^{i}D\right)K^{i+1}.
  \ee
  \State $i \leftarrow i+1$
  \Until{$\Vert Q^{i} - Q^{i-1}\Vert \leq \epsilon_{1}$}.
\end{algorithmic}
\end{breakablealgorithm}

\subsection{Analysis of Algorithm \ref{algo1}}
Now, we analyze the convergence, stability of the learned policy, and non-unique solutions of Algorithm \ref{algo1} using the following theorems.
\begin{theorem}[Convergence]
\label{conv_th}
Starting with an initial $Q^{0}$ such that $0<Q^{0}\leq \hat{Q}$, where $\hat{Q} >0$ is a solution to Theorem \ref{th1} associated with the $R>0$, Algorithm \ref{algo1} converges, and the solutions $Q^{i}, P^{i}, K^{i+1}~(i=0,1,2,\cdots)$ converge to $Q^{\ast}, P^{\ast}, K^{\ast}$ that satisfy
\be 
\label{riccth1}
P^{\ast} A+A^{\top} P^{\ast}+C^{\top} P^{\ast} C+Q^{\ast}-\left[P^{\ast} B+C^{\top} P^{\ast} D\right]\left[R+D^{\top} P^{\ast} D\right]^{-1}\left[B^{\top} P^{\ast}+D^{\top} P^{\ast} C\right]=0.
\ee
\be 
\label{opti_conth1}
K^{\ast} = K_{T} =  -\left(R+D^{\top}P^{\ast}D\right)^{-1}\left(B^{\top}P^{\ast} + D^{\top}P^{\ast}C\right).
\ee
Moreover, the solutions $Q^{\ast}, P^{\ast}$, and $K^{\ast}$ satisfy Theorem \ref{th1}.
\end{theorem}
\begin{proof}
\textbf{Convergence proof.} We first prove that $Q^{i}~(i=0,1,2,\cdots)$ solved by Algorithm \ref{algo1} is an increasing sequence. Substituting \eqref{con_algo1} for $K^{i}$ into \eqref{q_algo1} for $Q^{i}$, we obtain
\bex 
P^{i-1}A + A^{\top}P^{i-1}  + C^{\top}P^{i-1}C = (K^{i})^{\top}\left(R+D^{\top} P^{i-1}D\right)K^{i} - Q^{i},
\eex
which can be rewritten as 
\be 
\label{lya_pi-1}
\begin{aligned}
&  P^{i-1}\left(A+BK_{T}\right) + \left(A+BK_{T}\right)^{\top}P^{i-1}+ \left(C+DK_{T}\right)^{\top}P^{i-1}\left(C+DK_{T}\right) \\ &= (K^{i})^{\top}\left(R+D^{\top} P^{i-1}D\right)K^{i} - Q^{i} - K_{T}^{\top}\left(R+D^{\top} P^{i-1}D\right)K^{i} - (K^{i})^{\top}\left(R+D^{\top} P^{i-1}D\right) K_{T} + K_{T}^{\top}D^{\top} P^{i-1}D K_{T}.
\end{aligned}
\ee
Subtracting \eqref{lya_algo1} from \eqref{lya_pi-1} yields
\bex
\begin{aligned}
& \left(P^{i-1}-P^{i}\right)\left(A + BK_{T}\right) +\left(A + BK_{T}\right)^{\top}\left(P^{i-1}-P^{i}\right)+  \left(C+DK_{T}\right)^{\top}\left(P^{i-1}-P^{i}\right)\left(C+DK_{T}\right) \\ & = \left(K^{i} - K_{T}\right)^{\top}\left(R+D^{\top} P^{i-1}D\right)\left(K^{i} - K_{T}\right) \geq 0.
\end{aligned}
\eex
Since $K_{T}$ is a stabilizer of system \eqref{sde}, by \citep[Theorem 1]{rami2000linear}, it follows that $P^{i-1} \leq P^{i}$ holds for all iteration. Each pair of $(P^{i}, Q^{i+1})$ mentioned above satisfies \eqref{q_algo1}, and they correspond uniquely to each other. Now, according to \citep[Theorem 3.2.3]{sun2020stochastic}, it is known that if $Q^{i} \leq Q^{i+1}$, then $P^{i-1} \leq P^{i}$. For this reason, $Q^{i+1} \geq Q^{i} >0$ holds for $i=0,1,\cdots$, and $Q^{i+1} = Q^{i}$ holds if and only if $K^{i+1} = K_{T}$. Note that the goal of Algorithm \ref{algo1} is to achieve $K^{i+1} = K_{T} $.

Next, we claim that $Q^{i}$ is bounded by an upper bound. Let $\hat{P} >0$ and $\hat{Q} >0$ be a group of solutions that satisfy Theorem \ref{th1}. That is
\be 
\label{ricc_th2}
\hat{P}A + A^{\top}\hat{P} + C^{\top}\hat{P}C-\left(\hat{P}B+C^{\top}\hat{P}D\right)\left(R+D^{\top}\hat{P}D\right)^{-1}\left(B^{\top}\hat{P}+D^{\top}\hat{P}C\right) + \hat{Q} =0,
\ee
\be 
\label{KT_th2}
K_{T} = -\left(R+D^{\top}\hat{P}D\right)^{-1}\left(B^{\top}\hat{P} + D^{\top}\hat{P}C\right).
\ee

Applying \eqref{KT_th2}, we can rewrite \eqref{ricc_th2} as
\be 
\label{lya_th2}
\hat{P}\left(A + BK_{T}\right) + \left(A + BK_{T}\right)^{\top}\hat{P} + \left(C+DK_{T}\right)^{\top}\hat{P}\left(C+DK_{T}\right)+\hat{Q} + K_{T}^{\top}RK_{T} =0.
\ee
If $Q^{i} \leq \hat{Q}$ holds, then \eqref{lya_algo1} and \eqref{lya_th2} will solve $0 < P^{i} \leq \hat{P}$ with stabilizer $K_{T}$. With \eqref{opti_con}, \eqref{con_algo1}, and \eqref{KT_th2}, SAREs \eqref{q_algo1} and \eqref{ricc_th2} can be rewritten as 
\be 
\label{lyap_i+1th2}
 P^{i}\left(A+BK^{i+1}\right) + \left(A+BK^{i+1}\right)^{\top}P^{i}+ \left(C+DK^{i+1}\right)^{\top}P^{i} \left(C+DK^{i+1}\right) = -Q^{i+1} - (K^{i+1})^{\top}RK^{i+1},
\ee
\be 
\label{lyap_i+1Tth2}
\begin{aligned}
& \hat{P}\left(A+BK^{i+1}\right) + \left(A+BK^{i+1}\right)^{\top}\hat{P}+ \left(C+DK^{i+1}\right)^{\top}\hat{P} \left(C+DK^{i+1}\right) \\ &= -\hat{Q} + K_{T}^{\top}\left(R+D^{\top}\hat{P}D\right)K_{T} - (K^{i+1})^{\top}\left(R+D^{\top}\hat{P}D\right)K_{T} - K_{T}^{\top}\left(R+D^{\top}\hat{P}D\right)K^{i+1} + (K^{i+1})^{\top}D^{\top}\hat{P}DK^{i+1}.
\end{aligned}
\ee
respectively. Since \eqref{q_algo1} ensures that $K^{i+1}$ is a stabilizer of the system \eqref{sde}, subtracting \eqref{lyap_i+1th2} from \eqref{lyap_i+1Tth2} and using $0<P^{i}\leq \hat{P}$ gets
\bex 
\begin{aligned}
&\left(\hat{P}-P^{i}\right)\left(A+BK^{i+1}\right) + \left(A+BK^{i+1}\right)^{\top} \left(\hat{P}-P^{i}\right) + \left(C+DK^{i+1}\right)^{\top}\left(\hat{P}-P^{i}\right)\left(C+DK^{i+1}\right) \\ & = Q^{i+1} - \hat{Q} + \left(K^{i+1}-K_{T}\right)^{\top}\left(R+D^{\top}\hat{P}D\right)\left(K^{i+1}-K_{T}\right).
\end{aligned}
\eex
Consequently, $Q^{i+1}\leq \hat{Q}$ holds.

Through the above analysis, it can be concluded that if Algorithm \ref{algo1} is initialized with a $Q^{0}$ such that $0 < Q^{0} \leq \hat{Q}$, then $Q^{i}>0, i =0,1,\cdots$ will monotonically increase with an upper bound. Hence, Algorithm \ref{algo1} converges.

\textbf{Converged results.} We now show that converged solutions satisfy \eqref{riccth1} and \eqref{opti_conth1}.

Substituting \eqref{q_algo1} into \eqref{lya_algo1} gives
\bex 
\label{ricca_ii+1}
\begin{aligned}
&P^{i+1}A + A^{\top}P^{i+1} + C^{\top}P^{i+1}C + K_{T}^{\top}\left(B^{\top}P^{i+1}+D^{\top}P^{i+1}C\right) + \left(P^{i+1}B + C^{\top}P^{i+1}D\right)K_{T} + K_{T}^{\top}\left(R+D^{\top}P^{i+1}D\right)K_{T} \\ &=  P^{i}A + A^{\top}P^{i}+ C^{\top}P^{i}C - (K^{i+1})^{\top}\left(R+D^{\top}P^{i}D\right)K^{i+1}.
\end{aligned}
\eex
Taking $P^{i+1} = P^{i} =P^{\ast}$ as the convergence value, the above equation becomes
\be 
\label{KTKast}
\left(K^{\ast} - K_{T}\right)^{\top}\left(R+D^{\top}P^{\ast}D\right)\left(K^{\ast} - K_{T}\right) =0,
\ee 
where $K^{\ast} = -\left(R+D^{\top}P^{\ast}D\right)^{-1}\left(B^{\top}P^{\ast} + D^{\top}P^{\ast}C\right)$. Since $R+D^{\top}P^{\ast}D > 0$, \eqref{KTKast} yields $K^{\ast} = K_{T}$, which is actually \eqref{opti_conth1}. The converged $Q^{\ast}$ generated by \eqref{q_algo1} is exactly \eqref{riccth1}.

Then, we claim that the Converged solutions satisfy Theorem \ref{th1}.

Rewriting \eqref{riccth1} with \eqref{opti_conth1} yields
\bex 
 P^{\ast}\left(A+BK_{T}\right)+ \left(A+BK_{T}\right)^{\top}P^{\ast} + \left(C + DK_{T}\right)^{\top}P^{\ast}\left(C + DK_{T}\right) + Q^{\ast} + K_{T}^{\top}RK_{T} =0,
\eex 
which is actually \eqref{lya_KT}. Obviously, \eqref{riccth1} is actually \eqref{riccl}. Thus, $Q^{\ast}$, $P^{\ast}$, and $K^{\ast}$ satisfy Theorem \ref{th1}. This completes the proof.
\end{proof}

We next prove the stability of Algorithm \ref{algo1}.
\begin{theorem}[Stability]
Each iteration of  Algorithm \ref{algo1} yields a stabilizer $K^{i+1}$ by \eqref{con_algo1} for the learner agent \eqref{sdel}.
\end{theorem}
\begin{proof}
Since $K_{T}$ is a stabilizer of system \eqref{sde}, and $Q^{i} >0$, by \citep[Theorem 3.2.3]{sun2020stochastic}, Lyapunov Recursion \eqref{lya_algo1} exists a unique positive definite solution $P^{i} \in \mathbb{S}^{n}$.

Then, by the proof of \citep[Theorem 2.1]{li2022stochastic}, it is known that $K^{i+1}$ in \eqref{con_algo1} is also a stabilizer of system \eqref{sde}. From the proof of Theorem \ref{conv_th}, we know that $Q^{i+1} >0$, so $K^{i+1}$ in \eqref{con_algo1} is a stabilizer of system \eqref{sde} for all $i = 0,1,\cdots$.
\end{proof}

\begin{remark}
By \citep[Theorem 13]{rami2000linear}, the converged results $P^{\ast}$, $K^{\ast}$, and $Q^{\ast}$ solved by Algorithm \ref{algo1} shown in Theorem \ref{conv_th} obtain the optimal feedback control $u^{\ast} = K^{\ast}x$ of the stochastic LQ problem of the learner.
\end{remark}

In fact, the cost function weights $R$ and $Q^{\ast}$ satisfying \eqref{riccth1}-\eqref{opti_conth1} that obtain the same gain matrix
$K^{\ast} = K_T$ may not be unique. As a result, $R$ and $Q^{\ast}$ can be different from the expert's $R_{T}$ and $Q_{T}$, but $Q^{\ast}$ is an equivalent weight to $Q_{T}$. This is an important feature of the IRL algorithm. The following theorem will discuss the connection between $R_{T}, Q_{T}, P_{T}$ and $R,Q^{\ast}, P^{\ast}$.
\begin{theorem}[Non-uniqueness of solution]
Recall $R_{T}, Q_{T}, P_{T}$ satisfying \eqref{ricc} and \eqref{opti_con} and let $R_{o},Q_{o}, P_{o}$ satisfy
\be 
\label{BP_o}
B^{\top}P_{o} + D^{\top}P_{o}C = \left(R_{o}+D^{\top}P_{o}D\right)\left(R_{T}+D^{\top}P_{T}D\right)^{-1}\left(B^{\top}P_{T} + D^{\top}P_{T}C\right),
\ee
and
\be 
\label{ricc_o}
 P_{o}A + A^{\top}P_{o} + C^{\top}P_{o}C + Q_{o} - K_{T}^{\top}\left(R_{o}+D^{\top}P_{o}D\right)K_{T} = 0,
\ee
where $R_{o}=R_{T}-R$. Then any $P^{\ast} = P_{T}-P_{o}$ and $Q^{\ast} = Q_{T}-Q_{o}$ satisfy \eqref{riccth1} and \eqref{opti_conth1}.
\end{theorem}
\begin{proof}
Subtracting \eqref{ricc_o} from \eqref{ricc} gives
\be
\label{ricc_T-o} 
\left(P_{T} - P_{o}\right)A + A^{\top}\left(P_{T} - P_{o}\right) + C^{\top}\left(P_{T} - P_{o}\right)C + \left(Q_{T}-Q_{o}\right) - K_{T}^{\top}\left(R_{T}-R_{o} + D^{\top}\left(P_{T} - P_{o}\right)D\right)K_{T} = 0.
\ee
Using $P^{\ast} = P_{T} -P_{o}$, $R_{o} = R_{T} -R$ and $R>0$ in \eqref{BP_o} yields 
\bex 
\begin{aligned}
K^{\ast} =& \left(R+D^{\top}P^{\ast}D\right)^{-1}\left(B^{\top}P^{\ast} + D^{\top}P^{\ast}C\right) \\  =& \left(R+D^{\top}P^{\ast}D\right)^{-1}\left(B^{\top}P_{T} + D^{\top}P_{T}C\right) - \left(R+D^{\top}P^{\ast}D\right)^{-1}\left(R_{T}+D^{\top}P_{T}D-R-D^{\top}P^{\ast}D\right) \\ & \times \left(R_{T}+D^{\top}P_{T}D\right)^{-1}\left(B^{\top}P^{\ast} + D^{\top}P^{\ast}C\right) \\ = & \left(R_{T}+D^{\top}P_{T}D\right)^{-1}\left(B^{\top}P_{T} + D^{\top}P_{T}C\right) = K_{T},
\end{aligned}
\eex
which is \eqref{opti_conth1}. Substituting it into \eqref{ricc_T-o} gives \eqref{riccth1}. This completes the proof.
\end{proof}

\section{Model-free off-policy IRL algorithm}
In general, the knowledge of system dynamics may not be known. Therefore, we propose a model-free IRL algorithm based on the model-based Algorithm \ref{algo1}.

\begin{proposition}
\label{prop1}
Find $P^{i}$ from the matrix equation \eqref{lya_algo1} and obtain $K^{i+1}$ using \eqref{con_algo1} in Algorithm \ref{algo1} are equivalent to solving the following equation:
\be 
\label{ito_int}
\begin{aligned}
& \mathbb{E}^{\mathcal{F}_{t}}X(t+\Delta t)^{\top}P^{i}X(t+\Delta t) - x^{\top}P^{i}x - 2\mathbb{E}^{\mathcal{F}_{t}}\int_{t}^{t+\Delta t}\left(u(s)-K_{T}X(s)\right)^{\top}\tilde{B}^{i}X(s)ds \\ & -\mathbb{E}^{\mathcal{F}_{t}}\int_{t}^{t+\Delta t}u(s)^{\top}\tilde{D}^{i}u(s)ds + \mathbb{E}^{\mathcal{F}_{t}}\int_{t}^{t+\Delta t} (K_{T}X(s))^{\top}\tilde{D}^{i}K_{T}X(s)ds \\ = & -\mathbb{E}^{\mathcal{F}_{t}}\int_{t}^{t+\Delta t}X(s)^{\top}\left(Q^{i}+K_{T}^{\top}RK_{T}\right)X(s)ds,
\end{aligned}
\ee
where $\tilde{B}^{i} = B^{\top}P^{i}+D^{\top}P^{i}C$, $\tilde{D}^{i}=D^{\top}P^{i}D$, $\Delta t >0$ is the integral time period, and $X(\cdot)$ is governed by the learner agent \eqref{sdel} with any control $u(\cdot)$.
\end{proposition}
\begin{proof}
Firstly, we rewrite the learner agent \eqref{sdel} as 
\be 
\label{sdelu}
dX(s) = \left[AX(s) + BK_{T}X(s) + B\left(u(s)-K_{T}X(s) \right)\right]dt + \left[CX(s) + DK_{T}X(s) + D\left(u(s)-K_{T}X(s) \right)\right]dW(s).
\ee

By \eqref{lya_algo1} and \eqref{sdelu}, applying It\^{o}'s formula to $X^{\top}(s)P^{i}X(s)$ yields
\be 
\label{ito_1}
\begin{aligned}
d X^{\top}(s)P^{i}X(s) = - X(s)^{\top}\left(Q^{i}+K_{T}^{\top}RX(s)\right) + 2\left(u(s)-K_{T}X(s)\right)^{\top}\tilde{B}^{i}X(s) + u(s)^{\top}\tilde{D}^{i}u(s)-(K_{T}X(s))^{\top}\tilde{D}^{i}K_{T}X(s).
\end{aligned}
\ee
Then, integrating both sides of \eqref{ito_1} from $t$ to $t+\Delta t$ and taking conditional expectation $\mathbb{E}^{\mathcal{F}_{t}}$, we can get \eqref{ito_int}. This completes the proof.
\end{proof}

\begin{proposition}
  \label{prop2}
  Find $Q^{i+1}$ from the matrix equation \eqref{q_algo1} in Algorithm \ref{algo1} is equivalent to solving the following equation:
  \be
  \label{ito_int_Q} 
  \begin{aligned}
  &\mathbb{E}^{\mathcal{F}_{t}}\int_{t}^{t+\Delta t} X(s)^{\top}Q^{i}X(s)ds \\= &-\mathbb{E}^{\mathcal{F}_{t}}X(t+\Delta t)^{\top}P^{i}X(t+\Delta t) + x^{\top}P^{i}x + \mathbb{E}^{\mathcal{F}_{t}}\int_{t}^{t+\Delta t}(K^{i+1}X(s))^{\top}\left(R+\tilde{D}^{i}\right)K^{i+1}X(s)ds \\ & +2\mathbb{E}^{\mathcal{F}_{t}}\int_{t}^{t+\Delta t}u(s)^{\top}\tilde{B}^{i}X(s)ds + \mathbb{E}^{\mathcal{F}_{t}}\int_{t}^{t+\Delta t}u(s)^{\top}\tilde{D}^{i}u(s)ds.
  \end{aligned}
  \ee
\end{proposition}
\begin{proof}
Multiplying both sides of \eqref{q_algo1} by X and adding and subtracting terms $2u^{\top}B^{\top}P^{i}X$, $2u^{\top}D^{\top}P^{i}CX$ and $u^{\top}D^{\top}P^{i}Du$, \eqref{q_algo1} can be rewritten as 

\be 
\label{xQx}
\begin{aligned}
X^{\top}Q^{i}X =& -\left(AX+Bu\right)^{\top}P^{i}X - X^{\top}P^{i}\left(AX+Bu\right) - \left(CX+Du\right)^{\top}P^{i}\left(CX+Du\right) \\ & +(K^{i+1}X)^{\top}\left(R+\tilde{D}^{i}\right)K^{i+1}X + 2u^{\top}\tilde{B}^{i}X + u^{\top}\tilde{D}^{i}u,
\end{aligned}
\ee
where $u$ can be generated by any stabilizing policy.

Substituting \eqref{sdel} into \eqref{xQx}, integrating both sides of \eqref{xQx} from $t$ to $t+\Delta t$ and taking conditional expectation $\mathbb{E}^{\mathcal{F}_{t}}$, we can get \eqref{ito_int_Q}.
\end{proof}

For $P \in \mathbb{S}^{n}$, we define an operator as follows:
\bex
  svec(P) =\begin{bmatrix}p_{11}, 2 p_{12}, \ldots, 2 p_{1 n}, p_{22}, 2 p_{23}, \ldots, 2 p_{n-1, n}, p_{n n}\end{bmatrix}^{\top} \in \mathbb{R}^{\frac{1}{2}n(n+1)}.
\eex

By \citep{murray2002adaptive, li2022stochastic}, there exists a matrix $\mathcal{T} \in \mathbb{R}^{n^{2} \times \frac{1}{2}n(n+1)}$ with rank($\mathcal{T}) = \frac{1}{2}n(n+1)$ such that $vec(P) = \mathcal{T}\times svec(P)$ for any $P \in \mathbb{S}^{n}$. Then, we define $\bar{L} = \left(L^{\top} \otimes L^{\top}\right)\mathcal{T} \in \mathbb{R}^{m^{2} \times \frac{1}{2}n(n+1)}$, which satisfies $L^{\top} \otimes L^{\top} vec(P) = \bar{L}\times svec(P)$ for $L \in \mathbb{R}^{n \times m}$. For example, for vector $x \in \mathbb{R}^{n}$, we have 
\bex
\bar{x} =\begin{bmatrix}x_{1}^{2}, x_{1} x_{2}, \ldots, x_{1} x_{n}, x_{2}^{2}, x_{2} x_{3}, \ldots, x_{n-1} x_{n}, x_{n}^{2}\end{bmatrix}^{\top} \in \mathbb{R}^{\frac{1}{2}n(n+1)}.
\eex

For finding the parameters of $P^{i}$, $\tilde{B}^{i}$ and $\tilde{D}^{i}$ in \eqref{ito_int}, through vectorization and Kronecker product theory, we define the following operators,
\bex 
\begin{aligned}
\delta_{xx}  &= \mathbb{E}^{\mathcal{F}_{t}}\left[\bar{X}(t+\Delta t)-\bar{X}(t),~\dots,~\bar{X}(t+l\Delta t)-\bar{X}(t+(l-1)\Delta t) \right]^{\top};\\ \delta_{uu}  &= \mathbb{E}^{\mathcal{F}_{t}}\left[\int_{t}^{t+ \Delta t}\bar{u}(s)ds,~\dots,~\int_{t+(l-1)\Delta t}^{t+l\Delta t}\bar{u}(s)ds \right]^{\top}; \\ I_{xx}  &= \mathbb{E}^{\mathcal{F}_{t}}\left[\int_{t}^{t+ \Delta t}X(s)^{\top} \otimes X(s)ds,~\dots,~\int_{t+(l-1)\Delta t}^{t+l\Delta t}X(s)^{\top} \otimes X(s)ds \right]^{\top}; \\ 
I_{xu}  &= \mathbb{E}^{\mathcal{F}_{t}}\left[\int_{t}^{t+ \Delta t}X(s)^{\top} \otimes u(s)ds,~\dots,~\int_{t+(l-1)\Delta t}^{t+l\Delta t}X(s)^{\top} \otimes u(s)ds \right]^{\top};\\
\Phi_{p} &= \left[\delta_{xx},~2I_{xx}\left(I_{n}\otimes K_{T}^{\top}\right)-2I_{xu},~I_{xx}\bar{K}_{T}-\delta_{uu}\right]; \\
\Psi_{p}^{i} &= -\mathbb{E}^{\mathcal{F}_{t}}\left[\int_{t}^{t+\Delta t}X(s)^{\top}\left(Q^{i}+K_{T}^{\top}RK_{T}\right)X(s)ds,~\dots,~\int_{t+(l-1)\Delta t}^{t+l\Delta t}X(s)^{\top}\left(Q^{i}+K_{T}^{\top}RK_{T}\right)X(s)ds\right]^{\top}
\end{aligned}
\eex
where $l$ is the group number of sampling data and should satisfy $l \geq \frac{n(n+1)}{2}+mn+\frac{m(m+1)}{2}$.
For any $K_{T}$ obtained by by \eqref{ls_KT2}, \eqref{ito_int} implies
\bex
\Phi_{p}\begin{bmatrix} svec\left(P^{i}\right) \\  vec\left(\tilde{B}^{i}\right)\\svec\left(\tilde{D}^{i} \right)\end{bmatrix} =\Psi_{p}^{i}.
\eex

Then, using least squares method, $P^{i}$, $\tilde{B}^{i}$ and $\tilde{D}^{i}$ can be solved by
\be
\label{ls_p}
\begin{bmatrix} svec\left(P^{i}\right) \\  vec\left(\tilde{B}^{i}\right)\\svec\left(\tilde{D}^{i} \right)\end{bmatrix} =\left(\Phi_{p}^{\top}\Phi_{p}\right)^{-1}\Phi_{p}^{\top}\Psi_{p}^{i}.
\ee

Similarly, for \eqref{ito_int_Q}, we define
\bex 
\begin{aligned}
\Phi_{q} &= Ixx;\\
\Psi_{q}^{i+1} &= -\delta_{xx}svec\left(P^{i}\right)+I_{xx}\bar{K}^{i+1}svec\left(R+\tilde{D}^{i}\right) + 2I_{xu}vec\left(\tilde{B}^{i}\right) + 2\delta_{uu}svec\left(\tilde{D}^{i}\right).
\end{aligned}
\eex
Then, $Q^{i+1}$ can be uniquely calculated by
\be 
\label{ls_q}
svec(Q^{i+1}) = \left(\Phi_{q}^{\top}\Phi_{q}\right)^{-1}\Phi_{q}^{\top}\Psi_{q}^{i}.
\ee
By using \eqref{ls_p} and \eqref{ls_q}, we solve $P^{i}$, $Q^{i+1}$ and $K^{i+1} = -\left(R+\tilde{D}^{i}\right)^{-1}\tilde{B}^{i}$ using a data-driven method.

Next, we show that, under some rank conditions, $\Phi_{p}$ and $\Phi_{q}$ have full column rank. The proof of the following lemma is similar to the proof of \citep[Lemma 6]{jiang2012computational}, so we omit it.

\begin{lemma}
\label{lem_rank}
If there exists $l_{o} >0$, for all $l \geq l_{o} $,

\begin{flalign*}
  &rank([I_{xx},I_{xu}, \delta_{uu}]) = \frac{n(n+1)}{2} + mn + \frac{m(m+1)}{2}, \\ &
  rank(I_{xx}) =  \frac{n(n+1)}{2},
\end{flalign*}
then, \eqref{ls_p} and \eqref{ls_q} solve unique solutions, respectively.
\end{lemma}

Based on the above analysis, the model-free IRL algorithm is given in Algorithm \ref{algo2}.

\begin{breakablealgorithm}
  \label{algo2}
  \caption{Model-free IRL algorithm for the stochastic LQ problem.} 
  \renewcommand{\algorithmicrequire}{\textbf{Initialization:}}
\begin{algorithmic}[1]
  \Require Choose an arbitrary matrix $R >0$, an initial reward weight matrix $Q^{0}>0$, and compute $\hat{K}_{T}$ by \eqref{ls_KT2}, and collect system data generated by any stabilizing control input $u = K_{0}x+e$, where $e$ is the exploration noise.
  \State Set $i=0$.
  \Repeat
  \State (Policy correction) Update $P^{i}$ and control policy $K^{i+1}$ by \eqref{ls_p} and $K^{i+1} = -\left(R+\tilde{D}^{i}\right)^{-1}\tilde{B}^{i}$. 
  \State (Cost function weight construction) Update $Q^{i+1}$ by \eqref{ls_q}.
  \State $i \leftarrow i+1$
  \Until{$\Vert Q^{i} - Q^{i-1}\Vert \leq \epsilon_{1}$}.
\end{algorithmic}
\end{breakablealgorithm}

Finally, we prove the convergence of Algorithm \ref{algo2}.
\begin{theorem}
Under rank condition in Lemma \ref{lem_rank},  $\left\{P^{i}\right\}_{i=0}^{\infty}$, $\left\{Q^{i}\right\}_{i=0}^{\infty}$, $\left\{K^{i+1}\right\}_{i=0}^{\infty}$ converge to $P^{\ast}$, $Q^{\ast}$, $K^{\ast}$ that satisfy \eqref{riccth1} and \eqref{opti_conth1} in Algorithm \ref{algo2}.
\end{theorem}
\begin{proof}
Under the rank condition in Lemma \ref{lem_rank}, \eqref{ls_p} and \eqref{ls_q} have unique solutions $P^{i}$, $\tilde{B}^{i}$, $\tilde{D}^{i}$ and $Q^{i+1}$. Hence, according to Proposition \ref{prop1} and Proposition \ref{prop2}, Algorithm \ref{algo2} is equivalent to Algorithm \ref{algo1}. 
\end{proof}

\section{Numerical experiments}
In this section, we will give a simple example to verify the effectiveness of Algorithm \ref{algo2}. The system matrices $A$, $B$, $C$ and $D$ are
\bex
\begin{aligned}
  & A = \begin{bmatrix} -1&	2\\2.2&1.7 \end{bmatrix},\quad B= \begin{bmatrix} 2\\1.6 \end{bmatrix}, \quad C = \begin{bmatrix} 0.1&0.2\\0.2&0.1 \end{bmatrix},\quad D=\begin{bmatrix} 0.2\\0.1 \end{bmatrix}.
\end{aligned}
\eex

The actual target cost weights $R_{T}$ and $Q_{T}$ of the target agent \eqref{sdel}, corresponding target SARE solution $P_{T}$ and the target control gain $K_{T}$ are
\bex
\begin{aligned}
  & Q_{T} = \begin{bmatrix} 5&	0\\0&5 \end{bmatrix}, \quad R_{T} = 1,
\end{aligned}
\eex
\bex
\begin{aligned}
  & P_{T} = \begin{bmatrix} 0.8944	&0.0526\\
    0.0526	&2.1919 \end{bmatrix}, \quad K_{T} = \begin{bmatrix}-1.8279 &	-3.4648\end{bmatrix}.
\end{aligned}
\eex

Now, we will solve this problem using the presented model-free IRL algorithm. The initial cost weights $R$ and $Q^{0}$ are
\bex
\begin{aligned}
  & R = 1, \quad Q^{0} = \begin{bmatrix}0.2 &	0 \\ 0 & 0.2\end{bmatrix},
\end{aligned}
\eex

Then, we use $u = K_{0}x+e$ as a control input to collect data in the time interval [0,1], where $K_{0} = [-1.2292 ~~ -2.1684]$, and to ensure that $\Phi_{p}$ and $\Phi_{q}$ have full column rank, we make the exploration signal $e = 2 \sum_{i=1}^{10}sin(w_{i}s)$, where $w_{i}$, with i =1,\dots,10, are randomly selected from [-500,500].

With the initial states $x = [10 ~~ -10]^{\top}$, Figure \ref{fig1}  illustrates the iterative process of $P^{i}$, $K^{i}$, and $Q^{i}$ with stopping criterion $\left\| K^{i+1} - K_T\right\|<0.01$, where $P^{i}$, $K^{i}$ and $Q^{i}$ converge to
\bex
\begin{aligned}
  &P^{\ast} = \begin{bmatrix}0.3844 & 0.6872\\ 0.6872& 1.4023\end{bmatrix}, \quad Q^{\ast} = \begin{bmatrix}1.2006 &2.3116 \\2.3116 & 5.1567\end{bmatrix}, \quad K^{\ast} = \begin{bmatrix}-1.8236 & -3.4558\end{bmatrix}.
\end{aligned}
\eex

\begin{figure}[!htp]
  \centering 
  \includegraphics[width=1\textwidth]{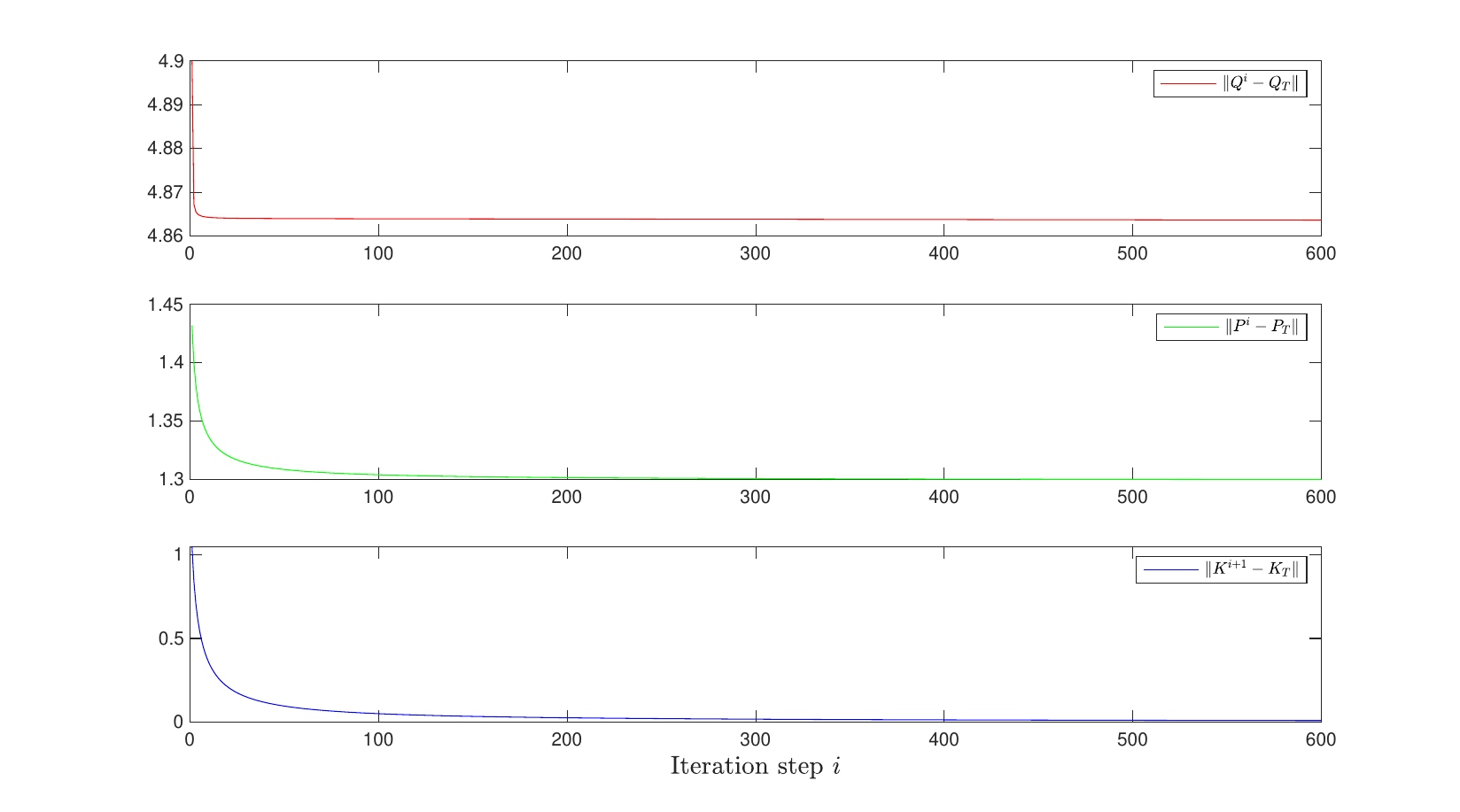}
  \caption{Convergence of $P^{i}$, $K^{i}$, and $Q^{i}$ using Algorithm \ref{algo2}.}
  \label{fig1} 
\end{figure}











\bibliographystyle{elsarticle-num-names} 
\bibliography{References}
\end{document}